\documentclass[11pt,a4paper, envcountsame,mathserif]{amsart}
\usepackage[usenames,dvipsnames]{color}

\usepackage[utf8]{inputenc}	
\usepackage{pdfpages,url}

 \usepackage[colorlinks,citecolor=blue,urlcolor=blue, linkcolor=blue, backref]{hyperref}

\usepackage[capitalise]{cleveref}		

 \usepackage{amsmath, amssymb, xspace, enumitem}
 \usepackage{graphicx}

 \usepackage[all]{xy}

\usepackage{tikz}
\usetikzlibrary{arrows,snakes,positioning,backgrounds,shadows}
\usepackage{stmaryrd}

\newtheorem{theorem}{Theorem}[section]

\newtheorem{thm}[theorem]{Theorem}

\newtheorem{prop}[theorem]{Proposition}

\newtheorem{claim}[theorem]{Claim}

\newtheorem{convention}[theorem]{Convention}

\newtheorem{lemma}[theorem]{Lemma}

\theoremstyle{definition}
\newtheorem{definition}[theorem]{Definition}
\newtheorem{example}[theorem]{Example}

\newcommand{\NN}{{\mathbb{N}}}

\newcommand{\sub}{\subseteq}
\newcommand{\sN}[1]{_{#1\in \NN}}

\newcommand{\uhr}[1]{\! \upharpoonright_{#1}}
\newcommand{\ML}{Martin-L{\"o}f}
\newcommand{\SI}[1]{\Sigma^0_{#1}}

\newcommand{\bi}{\begin{itemize}}
\newcommand{\ei}{\end{itemize}}
\newcommand{\bc}{\begin{center}}
\newcommand{\ec}{\end{center}}

\newcommand{\ES}{\emptyset}

\newcommand{\tp}[1]{2^{#1}}

\newcommand{\lep}{\le^+}

\newcommand{\seqcantor}{2^{ \NN}}

\newcommand{\cantor}{\seqcantor}
\newcommand{\strcantor}{2^{ < \omega}}

\newcommand{\ROpcl}[2]{[#2]^\prec_{#1}}

\newcommand{\n}{\noindent}

\newcommand{\sss}{\sigma}
\newcommand{\aaa}{\alpha}

\DeclareMathOperator{\Dim}{\textit{Dim}}
\newcommand{\lland}{\, \land \, }

\newcommand\+[1]{\mathcal{#1}}

\newcommand{\ol}{\overline}
\newcommand{\ul}{\underline}
\newcommand{\ape}{\, \hat{\ } \, }

\newcommand{\sssl}{\ensuremath{|\sigma|}}

\newcommand{\range}{\ensuremath{\mathrm{range}}}

\DeclareMathOperator{\diam}{diam}

\numberwithin{equation}{section}
\renewcommand{\hat}{\widehat}

\begin{document}

 \title{Fractal dimensions  and profinite groups}
\subjclass[2000]{} 
\author[E. Mayordomo]{Elvira Mayordomo}
\address{Departamento de Inform{\'a}tica e Ingenier{\'\i}a de Sistemas,
Instituto de Investigaci{\'o}n en Ingenier{\'\i}a de Arag{\'o}n,
Universidad de Zaragoza, 50018 Zaragoza, Spain}
\thanks{Mayordomo was supported
in part by Spanish Ministry of Science and Innovation grants
TIN2016-80347-R and PID2019-104358RB-I00  PID2022-138703OB-I00 and
by the Science dept. of Aragon Government: Group Reference
T64$\_$20R (COSMOS research group). }
\author[A. Nies]{Andr\'e Nies}
\address{School of Computer Science, University of Auckland, New Zealand}
 \thanks{Nies was supported in part by the Marsden fund of New Zealand, grant UoA-1931}
 \date{\today}
\maketitle

\begin{abstract} Let $T$  be  a finitely branching rooted tree such that any node has at least two successors. The path space $[T]$ is   an ultrametric space:  for  distinct paths $f,g$ let $d(f,g)= 1/|T_n|$, where  $T_n$ denotes the $n$-th level of the tree, and $n$ is largest  such that $f(n)= g(n)$. Let $S$ be a subtree of $T$ without leaves that is level-wise uniformly branching, in the sense that the number of successors of a node only depends on its level. We~show that the Hausdorff and lower box dimensions coincide for~$[S]$, and the packing and upper box dimensions also coincide. We give geometric proofs, as well as proofs based on the point-to-set principles. We use the first result to reprove  a theorem of Barnea and Shalev on the Hausdorff dimension of closed subgroups of a profinite group $G$, referring only on the geometric structure of the closed subgroup in the canonical   path space given  by an inverse system for $G$.  We obtain an analogous theorem for the packing dimension.
\end{abstract} 

\tableofcontents

\section{Introduction} We study fractal dimensions in the ultrametric setting. Our initial  motivation was  a description of the Hausdorff dimension of   a closed subgroup~$U$ of  a profinite group  $G$   by Barnea and Shalev \cite[Thm.\ 2.4]{Barnea.Shalev:97}, where the    metric is given by an inverse system $(L_n, p_n) \sN n$ for $G$. They showed that this dimension equals the lower box  dimension: \bc   $\dim_H(U) =\liminf_n \log |U_n|/\log |L_n|$,  \ec where $U_n$ is the canonical projection of $U$ into $L_n$.
We will show that  this equality can be obtained without making reference to  the group structure. Instead, one can    analyse the  fractal dimensions for   closed subsets of path spaces of finitely branching rooted trees that are  given by level-wise uniformly branching subtrees.  

 For the result on closed subsets of path spaces, we give  a geometric proof,  as well as a proof  using  the point-to-set principle in the form given in  \cite[Thm.\ 4.1]{Mayordomo.etal:23}, which   determines the dimension of a set via    an  algorithmic notion of dimension for individual points. In the same  two-pronged manner  we analyse  the packing dimension of such sets, and apply the result to profinite groups as well. 
 We include the   second  approach for three reasons.   Unlike the geometric  proofs, the proofs based on   point-to-set principles    use no preliminary results from the theory of fractional dimensions:  neither   the ``easy” inequalities~(\ref{eqn:dim inequs}), nor  Tricot’s result~\cite{Tricot:82} on packing dimension.
 Also, this line of argument promises to be extendable to settings less restrictive than the level-wise uniformly branching tree. Finally, this approach leads to a result of interest on its  own (Prop.\ \ref{cl:T3}) on  Martin-L\"of random paths on computable trees.

 \subsection{Path  spaces} Throughout, $T$ will denote a finitely branching rooted tree such that any node has at least two successors. For $n \ge 0$  we denote the $n$-level of $T$ by~$T_n$.  By a \textit{path} of $T$ one means a sequence $f = (a_0, a_1, \ldots)$ such that $a_n \in T_n$ and $a_{n+1} $ is a successor of $a_n$, for each $n$.  The set  $[T]$ of paths on $T$ is a totally disconnected compact space. The clopen sets are the  finite unions of sets $[\sss] = \{ f \colon    f(n) = \sss\}$, where  $\sss \in T_n$. 

 There is a natural ultrametric~$d$ on~$[T]$: For distinct $f,g \in [T]$, 
\begin{equation}  \label{eqn:distance tree} d(f,g)= \min  \{ 1/|T_n| \colon \, f(n) = g(n)\}.\end{equation}

Since every node $\sss$ on $T$ has at least two successors, the   set  $[\sigma]$ has diameter $1/|T_n|$ for $\sigma \in T_n$.

\subsection{Box  dimensions of closed subsets of $[T]$}  We will  study fractal dimensions of closed subsets of $[T]$.    The lower and upper box   (or Minkowski) dimension   of a bounded  set $X$ in a metric space~$M$ are
\begin{eqnarray} \label{eqn:lower box} \ul \dim_B(X)&=&
\liminf_{\alpha \to 0^+} \frac{\log N_\alpha(X)}{ \log (1/\alpha)} \\ 
\label{eqn:upper box} \ol \dim_B(X)&=&
\limsup_{\alpha \to 0^+} \frac{\log N_\alpha(X)}{ \log (1/\alpha)}  
\end{eqnarray}
where $N_\alpha(X)$ is the least cardinality  of a covering of $X$ with sets of diameter at most $\alpha$. 
 Now let  $M =[T]$ for a tree $T$ as above, and  write $t_n = |T_n|$. A~closed subset $X$ of $[T]$  has the form  $[S]$ for a unique subtree $S$ of $T$ without leaves.  
 If $t_{n-1}^{-1} > \alpha \ge t_n^{-1} $, then  we may assume that the sets in   a covering as above  are of the form $[\sss]$ for $\sss \in T_n$, because  each set of  diameter at most~$\aaa$ is contained in such a set $[\sss]$. So,  we have $N_\alpha(X)= |S_n|$. It follows that
\begin{eqnarray} \label{eqn:dimB}  \ul \dim_B([S])&=& \liminf_{n \to \infty} \frac{ \log |S_n|}{\log |T_n|}\\ 
\label{eqn:dimB2}  \ol \dim_B([S])&=& \limsup_{n \to \infty} \frac{ \log |S_n|}{\log |T_{n-1}|},   \end{eqnarray}
using that    $\log (1/ \alpha )\le \log |T_n|$ in (\ref{eqn:dimB}) and  $ \log |T_{n-1}|<  \log (1/ \alpha )$ in (\ref{eqn:dimB2}).

\subsection{Box  dimensions versus   Hausdorff/packing  dimension}   
 By $V^*$ one denotes the set of tuples over a set $V$, seen as a rooted tree where the successor relation is given by appending one element of $V$. A \emph{subtree} of $V^*$ is a nonempty subset closed under initial segments. We will always view a tree $T$ as a subtree of $\NN^*$. Thus, $T_k$ consists of the strings in $T$ that have  length $k$.
\begin{definition} \label{df:ubr}  We say that a tree $S \sub \NN^*$ is \emph{level-wise uniformly branching} if for each $n$, each $\sss \in S_n$ has the same number of successors on $S$. 
  \end{definition}

\begin{example} Let $T = \{0,1 \}^*$. Let $S$ be a  level-wise uniformly branching subtree of $T$ such that for  each even $k$, for each $n \in [4^k,4^{k+1})$, strings at level~$n$ have one successor, and for  each odd $k$, for each $n \in [4^k,4^{k+1})$, strings at level $n$ have two successors. Then $\ul \dim_B([S]) \le 1/4$ and $\ol \dim_B([S])\ge 3/4$.
\end{example}

For background on Hausdorff dimension   and also packing dimension $\dim_P$, see for instance the first two chapters of \cite{Bishop.Peres:17}.  Here we merely review  the pertinent definitions. Let  $X$ be a bounded  set in a metric space~$M$.

For $r\ge 0$, by $H^r(X)$ one  denotes the $r$-dimensional Hausdorff measure 
 \begin{equation} \label{Hdim} H^r(X) = \lim_{\aaa \to 0^+} ( \inf_{(E_i)} \sum_{i \in I} \diam(E_i)^r),\end{equation}  where the infimum is taken over all  covers $(E_i)_{i \in I}$  of $X$ by sets of diameter at most $\aaa$.  By definition, $\dim_HX$ is the supremum  of the $r$ such that $H^r(X) >0$. 
 
 For $r\ge 0$, by $P^r(X)$ one  denotes the $r$-dimensional packing measure
  \begin{equation}
      P^r(X)=\inf_{\mathcal{U}}\left\{\sum_{U\in\mathcal{U}} \lim_{\aaa\to 0^+}(\sup\sum_{i \in I}\diam(E_i^U)^r)\right\},           \end{equation}
 where the infimum is taken over all countable coverings $\+ U$ of $X$, and the   inner suprema are taken over all countable collections of disjoint open balls $(E_i^U)_{i \in I}$ with center in $U$ and diameter at most $\aaa$.  
Now   $\dim_P(X)$ is the supremum  of the $r$ such that $P^r(X) >0$.
 It is known that 
 \begin{equation}
\label{eqn:dim inequs} \dim_H(X) \le \ul \dim_B(X)  \text{ and }  \dim_P(X) \le \ol \dim_B(X)
 \end{equation}
   for any bounded subset $X$ of a metric space $M$  \cite[p.\ 72]{Bishop.Peres:17}.   

%
\section{Coincidences of dimensions}  Let $S$ be  a   level-wise uniformly branching subtree   of $T$.  View $[S]$ as a subset of the metric space~$[T]$. We    show   that the Hausdorff and lower box dimensions of $[S]$   coincide,  and  the  packing and upper box dimensions also  coincide.  This section  gives direct proofs of these facts.  The final section  provides alternative proofs    using  an algorithmic theory, in particular \ML\ randomness on $[S]$ and the point-to-set principles in metric spaces due to  Lutz et al.\ \cite{Mayordomo.etal:23}. Those proofs   directly obtain the equalities from the point-to-set principles, rather than relying on the   inequalities (\ref{eqn:dim inequs}).  

 

\begin{thm} \label{prop: dims equal} Let $S\sub T $ be trees such   that  each node on $T$ has at least two successors. Suppose  that $S$ is level-wise uniformly branching.  Then in the metric space $[T]$ we have  \begin{eqnarray*} \dim_H[S]  & =&  \ul \dim_B[S], \text{  and}   \\
\dim_P[S] &= &\ol \dim_B[S].\end{eqnarray*}
\end{thm}
\begin{proof}  It suffices to show the inequalities ``$\ge$", as the   inequalities  ``$\le$" always hold by  (\ref{eqn:dim inequs}).  We may assume that $T$ is a subtree of $ \NN^*$.
\medskip 

\n 1.  For the inequality $\dim_H [S] \ge \ul \dim_B [S]$, we may assume $\ul \dim_B [S]>0$. It suffices to  show  that $H^r[S]\ge 1$ for each   $r < \ul \dim_B [S]$ such that  $r > 0$.  

We call a tree $X \sub \NN^*$ \textit{left-closed} if $\sss \ape i \in X $ and $k < i$ implies $\sss \ape k \in X$. By the definition of the ultrametric (\ref{eqn:distance tree}), for each tree $Y$, there is a left-closed tree $X$ such that $[Y]$ is isometric to $[X]$. So we can assume that both $T$ and $S$ are left closed.

We first provide  a reduction of the type of coverings we need to consider. Write $t_n = |T_n| $ and $s_n= |S_n|$.  Suppose  we have a covering $(E_i)_{i \in I}$  of $[S]$ by sets of diameter at most $\aaa$, where $t_{n-1}^{-1} > \aaa \ge t_n^{-1}$.  In evaluating (\ref{Hdim})  we can assume that $\aaa = t_n^{-1}$.  The \emph{$r$-cost} of a family $(E_i)_{i \in I}$  is by definition
\bc $V_r(E_i)_{i \in I} =\sum_{i \in I} \diam(E_i)^r$. \ec 
 In  evaluating (\ref{Hdim}) we can further assume that  each $E_i$ intersects $[S]$. Also, without incurring a larger $r$-cost we can replace a covering of~$[S]$ by a  covering with sets that contain the original sets but have   no larger diameter. So  we can assume that each $E_i$ is of the form~$[\sss]$ for some $\sss \in S$, where $\sssl \ge n$. Since $[S]$ is compact and each set of the form $[\sss]$ is open, we can now assume that $I$ is finite. To summarise,   in  evaluating (\ref{Hdim}),  the only coverings of $[S]$ we need to consider are      coverings $\{ [\sss] \colon \sss\in F \}$ where $F \sub S$ is a finite prefix-free  set of strings of length at least $n$.  We will  identify such a set $F\sub S$ with the covering   it describes, and in particular write $V_r(F)$ for the $r$-cost of the  covering.  So $V_r(F) = \sum_{\rho \in F } t_{|\rho|}^{-r}$.
 
 \begin{claim} Given $n$, let $F \sub S$ be as above. Then $V_r(F) \ge s_kt_k^{-r}$ for some $k \ge n$.  \end{claim}
\n  To verify  this, we keep replacing $F$ by a  covering  $F'$ of $[S]$ of no larger $r$-cost,  until we reach a covering  of the form $\{[\sss] \colon \sss \in S_k\}$ for some $k \ge n$, which then has the required $r$-cost $s_kt_k^{-r}$. 
 
 Let $k\ge n$ be the least length of a string in $F$. For $\sss \in S_k$, write
 \[ v(\sss) = V_r\{ \rho \in F \colon \, \rho \succeq \sss\}. \]
 Note that $V_r(F)= \sum_{\sss \in S_k} v(\sss)$.   Pick $\sss \in S_k$ such that $v(\sss)$  is least. 
 
 \bi \item[(a)] If  $\sss \in F$, then $V_r(\{\sss\})= t_k^{-r}$, whence $V_r(F) \ge s_kt_k^{-r}$. So we end the process.
 \item[(b)]  Otherwise, let
 $ F' = \{ \eta \tau \colon \, \eta \in S_k \lland \sss \tau \in F\}$. Intuitively, above each $\eta \in S_k$, put a copy of what $F$ contains  above $\sss$.  
 Since $S$ is level-wise uniformly branching and left-closed, we have $F' \sub S$ and $F'$ is a covering of $[S]$. Also, $V_r(F') \le V_r(F)$, each  string in $F'$ has a length greater than $k$, $F'$ is prefix-free, and the maximum length of  strings in $F'$ is at most the one of $F$.  So we can continue the process with $F'$ instead of $F$.
 \ei 
After finitely many steps the process ends at (a), delivering the  covering of $[S]$ of the required form.  This shows the claim.

Since  $r < \ul \dim_B [S]$, there is $K \in \NN$ such that $\log s_k/\log t_k \ge r $ for each $k \ge K$, and hence $s_kt_k^{-r}\ge 1$.  For $\aaa < t_K^{-1}$, if  $(E_i)_{i \in I}$ is a covering of~$[S]$ such that $\diam (E_i) \le \aaa$ for each $i \in I$, then by the reductions and the  claim we have $V_r(E_i)_{i \in I} \ge  s_kt_k^{-r}$ for some $k \ge K$, and hence $V_r(E_i)_{i \in I} \ge 1$. Thus $H^r[S]\ge 1$ as required. 

\medskip 

\n 2.  $\dim_P [S] \ge \ol \dim_B [S]$. Our  proof     is based on 
   a  result of Tricot~\cite{Tricot:82} (also see \cite[Lemma 2.8.1(i)]{Bishop.Peres:17}). By that result,  since~$[S]$ as a   metric space is complete, 
\[\dim_P [S]  \ge \inf \{ \ol \dim_B ( V) \colon V  \neq \ES \text{ open in }  [S]\};\]
By the monotonicity of $\ol \dim_B$,   we can, without changing the  value,   take the infimum over all   sets  $V$ of the form $[\sss]\cap [S]$ where $\sss \in S$. Fix such a $\sss$, and write   $r = \sssl$. Since $S$ is level-wise uniformly branching, for $n \ge r$ the $n$-th level of the tree for $[\sss]\cap [S]$ has $s_n/s_r$ elements. So, using (\ref{eqn:dimB2}) we have \[\ol \dim_B(V) =  \limsup_n \frac {\log (s_n/s_r)}{\log t_{n-1}} = \limsup_n \frac {\log (s_n)}{\log t_{n-1}} = \ol \dim_B[S].\] 
\end{proof}

The following example shows that  some hypothesis on $S$, such as being level-wise uniformly branching,  is necessary in both equalities  of Theorem~\ref{prop: dims equal}.
\begin{example} Let $T= \strcantor$ be the full binary tree. For a   string $\sss \in T$, let $k_\sss = \min \{ i \colon \, i =  \sssl \lor   \sss_i = 1\}$. Let \bc $S = \{ \sss \in T \colon \, \sss_r= 0 \text{ for each } r \text{ such that } \sssl > r \ge 2 k_\sss\}$. \ec
Clearly $[S]$ is countable, and hence $\dim_H[S]=\dim_P[S] =0$. In contrast, since  $ 2 \cdot \tp{n/2}\ge |S_n| \ge \tp{n/2}$, we have $\ul \dim_B[S]= \ol \dim_B[S] = 1/2$.  
\end{example}

\section{Dimensions of closed subgroups  of profinite groups}

Let the topological  group \[G = \varprojlim \sN n \, (L_n, p_n)\]  be the inverse limit of an inverse system of discrete finite groups, where the maps $p_n \colon L_{n+1} \to L_n$ are epimorphisms, and  $L_0$ is the trivial group.   The group  $G$ is   compact  and has   a compatible  ultrametric, where the distance between distinct elements $x,y$ of $G$ is $1/|L_n|$ for the largest~$n$ such that the natural projections of $x$ and $y$  into $L_n$ are equal. 
 Barnea and Shalev~\cite[Thm.\ 2.4]{Barnea.Shalev:97} showed that  the Hausdorff dimension of a closed subgroup  $U$    of $G$ equals its lower box dimension.  They relied on Abercrombie~\cite[Prop.\ 2.6]{Abercrombie:94}  for     the nontrivial part, that the lower  box dimension is no larger than the Hausdorff dimension. Abercrombie proved this via means specific to profinite groups, such as generation of subgroups and  the Haar measure.  We will show   in Theorem~\ref{th:easy dim} that this  is a  purely geometric fact: it does not rely on  the group structure, but   merely exploits,   via Theorem~\ref{prop: dims equal},   the regularity of a description of  $U$  as a subtree of the tree given by the inverse system.   
 We also   obtain the  dual fact  that the packing dimension of a closed subgroup equals its upper box dimension.

 Define  a  tree $T$ such that  the $n$-th level is  $T_n= L_n$. A string $\eta \in T_{n+1}$ is a successor of  $\sss\in T_n$  if  $p_n(\eta)= \sigma$. Each $\sigma \in T_n$  has   $|L_{n+1}|/|L_n|$  successors.  So  $T$ is level-wise uniformly branching.  There is a canonical isometry between    $G$ and $[T]$:  map an element $x$   of $G$  to the   path $f_x$  such that   $f_x(n) = r$ if the natural projection map $G\to L_{n}$ sends $x$ to~$r$. 

\begin{thm}[\cite{Barnea.Shalev:97}, Thm.\ 2.4 for the first line of equations]  \label{th:easy dim}  \ \\  Suppose   that $U$ is a closed subgroup of $G = \varprojlim \sN n \, (L_n, p_n)$.  Let $U_n$  be the projection of~$U$ into $L_n$. Then \[ \dim_H(U) = \ul \dim_B(U) = \liminf_n \frac {\log |U_n|}{\log|L_n|}\]

 \[ \dim_P(U) = \ol \dim_B(U) = \limsup_n \frac {\log |U_n|}{\log|L_{n-1}|}.\]\end{thm}
\begin{proof}
	  We have $U = \varprojlim (U_n,q_n)$  where $q_n= p_n \uhr {U_{n+1}} \colon U_{n+1} \to U_n$ is  an onto homomorphism   for each $n$. Letting $S_n = \{ f\uhr n \colon f \in U\}$, $S$ is a subtree of~$T$, the elements of $U_n$ can be identified with $S_n$,   and    $U$ can be  isometrically identified with $[S]$.   As above, $S$ is level-wise uniformly branching. 
	 The result  
	 is now immediate from Theorem~\ref{prop: dims equal}.
\end{proof}

We end the section with some discussion of fractal dimensions, in particular in profinite groups. The following example shows that $\ul \dim_B(U) < \ol \dim_B(U)$ is possible. 
 \begin{example}[essentially \cite{Barnea.Shalev:97}, Lemma 4.1] \label{ex:group E} Let $E$ be the Cantor space $C_2^\omega$ with symmetric difference $\Delta$ as the group operation. Let the inverse system be given by $L_n= C_2^n$, with the canonical map $p_n$ that omits the first bit in a binary string of length  $n+1$. For each pair of reals $\aaa, \beta$ with $0 \le \aaa \le \beta \le 1$ there is a closed subgroup $U $ of $G$ such that    $\ul \dim_B(U) = \aaa$ and $\ol \dim_B(U) = \beta$.      \end{example} 
\begin{proof} Let $R \sub \NN$ be a set with lower [upper]  density  $\aaa$ [$\beta$]. Let $H $ be the subgroup  of sequences that have a $0$ in positions outside $R$.  We have $|S_n| = 2^{|R \cap n|}$, $|T_n|= 2^n$. Thus $\ul \dim_B(U) = \aaa$ and $\ol \dim_B(U) = \beta$ by (\ref{eqn:lower box}) and (\ref{eqn:upper box}), respectively.  \end{proof}

  If $G$ is a  $p$-adic analytic pro-$p$ group~\cite{Dixon.etal:03}, then by \cite[Thm 1.1]{Barnea.Shalev:97},  \bc $\dim_H(U)=\dim_{\text{manifold}}(U) /\dim_{\text{manifold}}(G)$, \ec  where $\dim_{\text{manifold}}(X)$ is the dimension of a closed set $X$ in the $p$-adic manifold underlying $G$.  Barnea and Shalev~\cite[end of Section 3]{Barnea.Shalev:97}    assert that the proof  also works for the upper box dimension in place of $\dim_H$, so the same expression holds. Thus all the fractal dimensions agree for closed subgroups of $G$; in particular,  $\dim_H$ and $\dim_P$ agree.

The \emph{spectrum} of a profinite group is the set of Hausdorff dimensions of closed subgroups.   The group  $E$ given in Example~\ref{ex:group E} has the largest possible  spectrum $[0,1]$. Note that $E$ is not (topologically) finitely generated.   Recent  research addresses questions about  which spectra can occur for finitely generated groups, with particular constraints on the groups and the type of subgroups.  For instance,  de las Heras and Klopsch~\cite{delasheras2022prop}  proved that there is a 2-generated  pro-$p$ group $G$ such that with respect to any of the five standard ways for pro-$p$ groups to represent  it as an inverse limit,  each real in $[0,1]$ occurs as the dimension of a \textit{normal} closed subgroup.

 \section{Algorithmic  dimension of points} 
\n  The remainder of the paper is devoted to the \emph{algorithmic} theory of fractal dimensions. The  present section   reviews   the Kolmogorov complexity of points in  metric spaces, and describes the point-to-set principle for both Hausdorff and packing dimension  in that context. We mostly follow Lutz, Lutz and Mayordomo~\cite{Mayordomo.etal:23}. However, we   restrict ourselves   to  the standard gauge family $\theta_s(\aaa) = \aaa^s$, which has a precision family in their sense  \cite[Observation 2.1]{Mayordomo.etal:23}.  For   further background     see  \cite{Mayordomo:08}. 

\begin{convention} All logarithms will be taken with respect to base $2$. \end{convention}

 Suppose one is  given a     metric space $M$ together with   a countable  dense set $D\sub M$ and    an onto  function $r \colon \{0,1\}^* \to D$.
  For $w\in D$ one  lets \begin{equation} \label{eqn:  CCC}  C^{M,r}(w)= \min \{ C(u) \colon r(u) = w\}, \end{equation}   
  where $C(u)$ is the usual Kolmogorov complexity of a binary string $u$. (See Li~and Vitanyi~\cite{Li.Vitanyi:93}, the standard reference for   Kolmogorov  complexity.) %
 \begin{definition}[\cite{Mayordomo.etal:23}, (3.1) and (3.2)] \label{df:ccdim} For  $x\in M$ and any $\aaa > 0$,  
 one defines the \emph{Kolmogorov complexity of a point $x$ at precision   $\alpha$} by  \begin{equation} \label{eqn: define Calpha}  C_\alpha(x) =
\min \{C^{M,r}(w) \colon  w \in D \lland d(x,w)\le \alpha\}.\end{equation}

 The \emph{algorithmic  dimension} of $x \in M$ is 
 \begin{equation} \label{eqn:constr dim}  \textit{dim}(x) :  = \liminf_{\alpha \to 0^+} C_\alpha(x)/ \log (1/\alpha).\end{equation}
  The \emph{strong algorithmic  dimension} of $x \in M$ is 
 \begin{equation} \label{eqn:constr dim}  \textit{Dim}(x) :  = \limsup_{\alpha \to 0^+} C_\alpha(x)/ \log (1/\alpha).\end{equation} \end{definition} 
  These definitions can be readily relativised to an oracle set $A \sub \NN$:  in~(\ref{eqn:  CCC}), replace $C(u) $ by $C^A(u)$, the Kolmogorov complexity of a string~$u$ as measured by a universal machine equipped with $A$ as an oracle.
 Two \emph{point-to-set principles}     \cite[Thms 4.1,4.2]{Mayordomo.etal:23}   determine the Hausdorff and packing dimension of a \emph{set} $X \sub M$, respectively,  in terms of the relativised algorithmic  dimensions of the\emph{ points} in it:
 \begin{eqnarray} \label{eqn:PTSP}  \dim_H (X) & = &  \min_A \sup_{x \in X} \textit{dim}^A(x),\\ \dim_P (X) & = &  \min_A \sup_{x \in X} \textit{Dim}^A(x) \label{eqn:PTSP2}  \end{eqnarray}

\begin{definition}  Fix some computable injection   $\pi \colon   \NN^* \to \{0,1\}^*$, say \bc $(n_0, \ldots, n_{k-1}) \mapsto 0^{n_0 } 1 \ldots 0^{n_{k-1}} 1$.   \ec    The plain descriptive (Kolmogorov) complexity $C(y) $ of $ y \in \NN^*$    is defined to be $C (\pi(y))$;  in a similar way  one defines $K(y)$, $C(y\mid n)$, etc. \end{definition} 
The choice of $\pi$ will only affect these values up to a   constant.
It will be  clear from the context whether $C(y)$ is taken to be evaluated for $y \in \NN^*$ or for $y \in \{0,1\}^*$. There are now  two interpretations of the expression $C(y)$ for    a binary string $y$; however, they agree up to  a constant. 

 
 Recall that throughout,  $T \sub \NN^*$ is a finitely branching  tree such that each node has at least two successors. By $T_n$ we denote the set of strings in $T$ of length $n$. The space  $[T]$   carries the   ultrametric 
 given by~(\ref{eqn:distance tree}).  For real-valued expressions $E,F$ depending on the same parameters,  we write $E \le^+F$ for $E \le F + \+ O(1)$, and $E=^+ F$ for $E \le^+ F \le^+ E$.

  Suppose $T$ is computable. We     describe a computable structure on the metric space $[T]$ in the usual sense of computable analysis~\cite{brattka2010cca}.  That is, we describe a function $r \colon \{ 0,1\}^*  \to D$ as above such that the distance of $r(u)$ and $r(v)$ is a computable real, uniformly in strings $u, v$.   Let $\ul r \colon \{ 0,1\}^* \to T$ be the   computable bijection that is an isomorphism between  the length lexicographical orderings on $\{ 0,1\}^*$ and~$T$. Given a binary string~$u$, let  $r(u)$ be  a path on $[T]$ extending the string  $ \ul r(u)$. Since $T$ is computable, one  can choose $r(u)$ to be   uniformly computable in~$u$. Clearly  the distances are computable reals as required above. 
  
We next relate $C_\aaa(x)$, the Kolmogorov complexity of a point at precision~$\aaa$  in   (\ref{eqn: define Calpha}), to the initial segment  complexity of paths on $[T]$.
 \begin{prop} \label{cl: T1} Suppose that     $T\sub \NN^*$ is a computable tree such that each node has at least two successors. Let $r \colon \{ 0,1\}^*\to [T]$ define a computable structure as above. Let $f \in [T]$.  For each~$\aaa$ such that  $|T_{n-1}|^{-1} >\alpha \ge |T_n|^{-1}$, we have  \bc $ C(f\uhr n) - 2 \log n \lep  C(f\uhr n \mid n) \lep C_{\alpha}(f)\lep   C(f\uhr  n)$. \ec \end{prop}
\begin{proof}  The leftmost inequality is standard~\cite{Li.Vitanyi:93}. For the middle inequality,  
   suppose    $w \in D = \range(r)$ and  $d (w, f) \le \aaa$. Given $u \in  \{ 0,1\}^* $ such that $r(u)= w$ and $n\in \NN$, we can compute $w \uhr n$,  which   equals $f\uhr n$ since $|T_{n-1}|^{-1} >\alpha$.  Thus $ C(f\uhr n \mid n) \lep C(w)$. 
 Since $w$ was arbitrary this shows $C(f\uhr n \mid n)  \lep C_\alpha(f)$. 
 
 For the rightmost inequality,   given a string $y = f\uhr n \in T$,   compute the  string $u\in \{0,1\}^*$ such that $\ul r(u)= y$. Since $d(f, r(u)) \le \alpha$, this shows $C_\aaa(f) \lep C(u)\le^+ C(y)$. \end{proof}

 The algorithmic dimensions  $\textit{dim}(f)$ and $\textit{Dim}(f)$ (see Def.\ \ref{df:ccdim}) of a path $f \in [T]$ can be expressed in terms of the descriptive complexity of its initial segments. This is  well-known at least for the case of a binary tree $T$, going back to~\cite{Mayordomo:02}.
\begin{prop}  \label{cl: T2} Let $T$ be a computable tree such that   each string on $T$ has at least two successors on $T$. For any  $f\in [T]$   we have  \[\textit{dim}(f)= \liminf_n \frac{C(f \uhr n) }{ \log |T_n|} = \liminf_n \frac{K(f \uhr n) }{ \log |T_n|}\] 
 \[\textit{Dim}(f)= \limsup_n \frac{C(f \uhr n) }{ \log |T_{n-1}|} = \limsup_n \frac{K(f \uhr n) }{ \log |T_{n-1}|}\] 
\end{prop}
\begin{proof}   For the left  equalities,  recall that  for $|T_{n-1}|^{-1} >\alpha \ge |T_n|^{-1}$ we have $ C(f\uhr n) - 2 \log n \lep    C_{\alpha}(f)\lep   C(f\uhr  n)$  by Proposition~\ref{cl: T1}. By hypothesis   $\log|T_n| \ge n$, and so $\lim_n \log n/\log |T_n|=0$.  Therefore 
\begin{eqnarray*}  \liminf_{\alpha \to 0^+} C_\alpha(f)/ \log (1/\alpha)&=& \liminf_{n \to \infty} C(f\uhr n)/ \log |T_n| \\
   \limsup_{\alpha \to 0^+} C_\alpha(f)/ \log (1/\alpha)&= &\limsup_{n \to \infty} C(f\uhr n)/ \log |T_{n-1}| \end{eqnarray*}
For the right  equalities, note that for each $n$ and each $y \in T_n$, we have     $C(y) \lep K(y) \lep C(y) + K(n)\lep C(y) + 2 \log n$. 
\end{proof}

\section{\ML\ randomness and Levin-Schnorr Theorem on $[S]$.}
  \begin{definition} \label{def: unif measure on T}
	For a  finitely branching tree $S$ without leaves,  the path space	$[S]$ carries the  uniform     measure $\mu_S$ determined  by  the conditions that  \bi \item  $\mu_S[S]=1$  \item    $\mu_S([\sss\ape i]\cap [S])= \mu_S([\sss]\cap [S])/k_\sss$ where $k_\sss$ is the number of successors of $\sss$ on $S$.   \ei Note that  if $S$ is   level-wise uniformly branching, then $\mu_S([\sss]) = |S_n|^{-1}$ for each $\sss \in S_n$. 
	\end{definition} 
		If $S$ and the function  $\sss \mapsto k_\sss$ are computable, then $\mu_S$ is computable in the usual sense of~\cite{hoyrup2009computability}.  Thus,   algorithmic test notions  such as \ML\ tests (ML-tests) can be used for $[S]$ in place of Cantor space. Also, one can generalise the  Levin-Schnorr Theorem (see, e.g.,  \cite[3.2.9]{Nies:book}) to $[S]$.   For $X \sub S$ let $\ROpcl {S} X =  \{ f \in S \colon \exists n  \, f \uhr n \in X\}$ be the open subset of $[S]$ described by $X$.
 Let \[ R_{S,b} = \ROpcl {S} { \{ \sss \in S \colon K(\sss) \le -  \log \mu_S[\sss] - b\}}.\] 
 Our version of the Levin-Schnorr theorem is as follows.  
 \begin{lemma} \label{lem:SL}  $(R_{S,b})\sN b$ is a universal \ML\ test on $[S]$.  Thus, it is a ML-test, and $f\in [S] $ is ML-random iff there is $b\in \NN$ such that \bc $K(f\uhr n) > - \log \mu_S[f \uhr n ] - b$ for each~$n\in \NN$. \ec
 \end{lemma}
 \begin{proof} Clearly the   sets $R_{S,b}$ are  $\SI 1$ uniformly in $b$. Given $b$, let  $D$ be the set of prefix minimal strings $\sss \in S$ such that $K(\sss) \le - \log \mu([\sss])-b$. Then \[ 1 \ge \sum_{\sss \in D} \tp{-K(\sss) } \ge \sum_{\sss \in D}\tp{\log \mu([\sss]) + b}.\]
 This implies that $\mu_S \ROpcl S D \le \tp {-b}$.  
 
Universality is shown as in the proof for Cantor space $\cantor$; see \cite[3.2.9]{Nies:book}.  To adapt that proof, one first needs to write down a machine existence theorem for machines with binary input that output strings  in $\NN^*$. As this is well-known in principle, we omit the details.
 \end{proof} 
 
 While we are mainly interested in the case that  $[S]$ is uncountable,  the argument above also works  for countable $[S]$. 
In that   case, only the empty set is null, so every path is ML-random on $[S]$.
 \medskip

 \section{Algorithmic  dimension versus box dimension}

We show that in the case that a  computable subtree $S$ of $T$  is level-wise uniformly branching, the lower/upper  box dimension  of $[S]$ equals the  (strong) algorithmic  dimension of  a   ML-random   $f \in [S]$. To provide the desired proof of  Theorem~\ref{prop: dims equal}  using  the point-to-set principles,  this fact  will be used in relativised form.

\begin{prop} \label{cl:T3} \ 

\n  (i) Let $T$  be  a computable tree as above, and let  $S$ be a computable subtree of $T$.  For each    $f \in [S]$, we have 
\begin{eqnarray*} \textit{dim}(f) &\le &  \ul \dim _B([S]), \text{ and}  \\
\textit{Dim}(f) &\le&   \ol \dim _B([S]).    \end{eqnarray*}

\n (ii) Suppose $S$ is level-wise uniformly branching. Then the equalities hold in the case  that $f$  is \ML\ random in $[S]$ with respect to~$\mu_S$.  \end{prop} 
\begin{proof} Throughout we will use (\ref{eqn:dimB}),(\ref{eqn:dimB2}) and  Prop.~\ref{cl: T2}.  Let $S_{\le n} = \bigcup_{k\le n} S_k$. 

\n (i)  Let $\eta_S \colon S \to \{0,1\}^*$ be    the  (computable) isomorphism between the  length-lexicographical orderings on $S$ and on $\{0,1\}^*$.  
Recall that   to define $C(y)$ and $K(y)$ for $y \in \NN^*$, we fixed a  computable injection $\pi \colon \NN^* \to \{0,1\}^*$. Also recall that logarithms are taken   in base $2$.

\smallskip 

\n {\bf Claim.}  {\it  Let   $y \in S_n$. Then $C(y)  \le^+ \log |S_{\le n} |$.   }

\smallskip

\n To verify  this, 
let $x = \eta_S(y)$.  We have     $|x| \le \lceil \log |S_{\le n}|\rceil$.  Since $x$ is binary,   $C(x) \le^+ |x| $. Since   the function $\pi \circ \eta_S^{-1}$ is computable, we have  $C(y) = C(\pi (\eta_S^{-1}(x)))\lep C(x)$. The claim follows. 

To verify  the  inequalities (i), let $y = f\uhr n$.  Note that \bc $\log |S_{\le n} | \le \log (n \cdot  |S_n|) \lep \log n + \log |S_n|$. \ec So by the claim, $C(y) \lep \log n  + \log |S_{n} |$. Using   that     $\log n   = o(\log |T_n|) $, 

\begin{eqnarray*} \dim(f) &=&  \liminf_n \frac{C(f\uhr n)} {\log |T_n| }  \\ 
& \le&  \liminf_n \frac{ \log n + \log |S_n|}{ \log |T_n| }= \ul \dim_B[S]. \end{eqnarray*}
 Similarly, 
 \begin{eqnarray*} \Dim(f) &=&  \limsup_n \frac{C(f\uhr n)} {\log |T_{n-1}| }  \\ 
 	& \le&  \limsup_n \frac{ \log n + \log |S_n|}{ \log |T_{n-1}| }= \ol  \dim_B[S]. \end{eqnarray*}

%
%
%

\n (ii) If $f$ is ML random then $f$ passes the universal test $(R_{S,b}) \sN b$ in  Lemma~\ref{lem:SL}. Thus, using that $S$ is level-wise uniformly branching, we have   \bc $\log |S_n| =^+ -\log \mu_S[f\uhr n]  \lep K(f\uhr n)$.  \ec
The equalities now follow from   Proposition~\ref{cl: T2}.  
\end{proof}

\begin{proof}[Proof of Th.\ \ref{prop: dims equal} using the point-to-set principle] 
Suppose   that $E\sub \NN$ is an oracle such that $S \le_T E$, $T \le_T E$, and $E$ can compute  from~$n$ the   number  of successors for $S$ at   level~$n$. Note that  the uniform measure $
\mu_S$ on $[S]$ is $E$-computable.

We  show $\dim_H[S] =  \ul \dim_B[S]$.  By      the point-to-set principle for Hausdorff dimension  (\ref{eqn:PTSP}),  it suffices to show that $\min_A \sup_{f \in [S]} \textit{dim}^A(f)= \ul \dim_B([S])$. We have the inequality ``$\le$"  by Prop.\ \ref{cl:T3}(i) relativised to $E$. For the converse inequality, note that 
in taking the minimum on the left hand side, we may assume that the oracle $A$ computes $E$, because replacing $A$ by $A \oplus E$  does not increase  $\textit{dim}^A(f)$. For each such $A$, if we choose $f $ to be  ML$^A$-random in $[S]$, by Prop.~\ref{cl:T3}(ii) we have  $\textit{dim}^A(f)=  \ul \dim_B(S)$.

The second  equality  in Th.\ \ref{prop: dims equal} is obtained in a similar fashion,  using  the point-to-set principle  for packing dimension (\ref{eqn:PTSP2}) instead.
\end{proof}

\def\cprime{$'$} \def\cprime{$'$}

\end{document}